\documentclass[12pt]{article}
\usepackage{srcltx}
\usepackage{eurosym}
\usepackage{mathtools}
\usepackage{amsmath}
\usepackage{amsfonts}
\usepackage{amssymb}
\usepackage{amsthm}
\usepackage{graphicx}
\usepackage{mathrsfs}

\usepackage{exscale}
\usepackage{latexsym}
\usepackage{authblk}

\usepackage{bbm}

\usepackage{xcolor}
\usepackage{tocloft}

\usepackage[colorlinks,plainpages=true,pdfpagelabels,hypertexnames=true,colorlinks=true,pdfstartview=FitV,linkcolor=blue,citecolor=red,urlcolor=black]{hyperref}
\PassOptionsToPackage{unicode}{hyperref}
\PassOptionsToPackage{naturalnames}{hyperref}
\usepackage{enumerate}
\usepackage[shortlabels]{enumitem}
\usepackage{bookmark}
\usepackage{wasysym}
\usepackage{esint}
\usepackage[ddmmyyyy]{datetime}
\usepackage[margin=1in]{geometry}
\numberwithin{equation}{section}
\everymath{\displaystyle}
\usepackage[capitalize,nameinlink]{cleveref}
\crefname{section}{Section}{Sections}
\crefname{subsection}{Subsection}{Subsections}
\crefname{condition}{Condition}{Conditions}
\crefname{hypothesis}{Hypothesis}{Conditions}
\crefname{assumption}{Assumption}{Assumptions}
\crefname{lemma}{Lemma}{Lemmas}
\crefname{definition}{Definition}{Definitions}

\crefformat{equation}{\textup{#2(#1)#3}}
\crefrangeformat{equation}{\textup{#3(#1)#4--#5(#2)#6}}
\crefmultiformat{equation}{\textup{#2(#1)#3}}{ and \textup{#2(#1)#3}}
{, \textup{#2(#1)#3}}{, and \textup{#2(#1)#3}}
\crefrangemultiformat{equation}{\textup{#3(#1)#4--#5(#2)#6}}%
{ and \textup{#3(#1)#4--#5(#2)#6}}{, \textup{#3(#1)#4--#5(#2)#6}}%
{, and \textup{#3(#1)#4--#5(#2)#6}}

\Crefformat{equation}{#2Equation~\textup{(#1)}#3}
\Crefrangeformat{equation}{Equations~\textup{#3(#1)#4--#5(#2)#6}}
\Crefmultiformat{equation}{Equations~\textup{#2(#1)#3}}{ and \textup{#2(#1)#3}}
{, \textup{#2(#1)#3}}{, and \textup{#2(#1)#3}}
\Crefrangemultiformat{equation}{Equations~\textup{#3(#1)#4--#5(#2)#6}}%
{ and \textup{#3(#1)#4--#5(#2)#6}}{, \textup{#3(#1)#4--#5(#2)#6}}%
{, and \textup{#3(#1)#4--#5(#2)#6}}

\crefdefaultlabelformat{#2\textup{#1}#3}

\newtheorem{theorem} {Theorem}[section]
\newtheorem{proposition} [theorem]{Proposition}

\newtheorem{corollary}[theorem]{Corollary}
\newtheorem{example}[theorem]{Example}
\newtheorem{question}[theorem]{Question}
\newtheorem{counter example}[theorem]{Counter Example}
\newtheorem{remark}[theorem] {Remark}

\def\CC{{\rm \kern.24em \vrule width.02em height1.4ex depth-.05ex \kern-.26emC}}

\def\TagOnRight

\def\AA{{it I} \hskip-3pt{\tt A}}

\def\QQ{\rlap {\raise 0.4ex \hbox{$\scriptscriptstyle |$}} {\hskip -0.1em Q}}

\makeatletter
\newcommand{\vo}{\vec{o}\@ifnextchar{^}{\,}{}}
\makeatother

\def\YYint#1#2#3{{\setbox0=\hbox{$#1{#2#3}{\iint}$}
		\vcenter{\hbox{$#2#3$}}\kern-.50\wd0}}

\def\XXint#1#2#3{{\setbox0=\hbox{$#1{#2#3}{\int}$}
		\vcenter{\hbox{$#2#3$}}\kern-.50\wd0}}

\makeatletter
\def\namedlabel#1#2{\begingroup
	\def\@currentlabel{#2}%
	\label{#1}\endgroup
}
\makeatother
\makeatletter
\newcommand{\rmh}[1]{\mathpalette{\raisem@th{#1}}}
\newcommand{\raisem@th}[3]{\hspace*{-1pt}\raisebox{#1}{$#2#3$}}
\makeatother

\newcounter{desccount}

\newcommand{\descref}[2]{\hyperref[#1]{\textnormal{\textcolor{black}{}\textcolor{blue}{ #2}\textcolor{black}{}}}}

\newcommand{\dref}[2]{\hyperref[#1]{\textcolor{black}{(}\textcolor{blue}{\bf #2}\textcolor{black}{)}}}
\newcommand{\be} {\begin{eqnarray}}
	\newcommand{\ee} {\end{eqnarray}}
\newcommand{\Bea} {\begin{eqnarray*}}
	\newcommand{\Eea} {\end{eqnarray*}}
\newcommand{\pa} {\partial}

\newcommand{\rr}{\rightarrow}

\newcommand{\B} {\beta}

\newcommand{\la} {\lambda}

\newcommand{\f}{\infty}
\newcommand{\R}{\mathbb{R}}

\newcommand{\noi} {\noindent}

\newcommand{\norm}[1]{\left|\hspace{-0.2mm}\left| #1 \right|\hspace{-0.2mm}\right|}
\newcommand{\abs}[1]{\left| #1\right|}

\newcounter{whitney}
\refstepcounter{whitney}

\newcounter{ineqcounter}
\refstepcounter{ineqcounter}
\makeatletter
\def\ps@pprintTitle{%
	\let\@oddhead\@empty
	\let\@evenhead\@empty
	\def\@oddfoot{}%
	\let\@evenfoot\@oddfoot}
\makeatother
\usepackage[titletoc,toc,page]{appendix}

\makeatletter
\newcommand{\refcheckize}[1]{%
	\expandafter\let\csname @@\string#1\endcsname#1%
	\expandafter\DeclareRobustCommand\csname relax\string#1\endcsname[1]{%
		\csname @@\string#1\endcsname{##1}\wrtusdrf{##1}}%
	\expandafter\let\expandafter#1\csname relax\string#1\endcsname
}
\makeatother

\refcheckize{\cref}
\refcheckize{\Cref}


\makeatletter
\newcommand{\mainsectionstyle}{%
	\renewcommand{\@secnumfont}{\bfseries}
	\renewcommand\section{\@startsection{section}{2}%
		\z@{.5\linespacing\@plus.7\linespacing}{-.5em}%
		{\normalfont\bfseries}}%
}
\makeatother
\usepackage{pgf,tikz}
\usetikzlibrary{arrows}
\usetikzlibrary{decorations.pathreplacing}

\usepackage{xpatch}
\xpatchcmd{\MaketitleBox}{\hrule}{}{}{}
\xpatchcmd{\MaketitleBox}{\hrule}{}{}{}

\date{}

\usepackage{scalerel}

\makeatletter

\title{}

\allowdisplaybreaks

\usepackage{accents}
\newlength{\dhatheight}

\title{A note on existence of smooth solution to Jacobian equation for compactly supported smooth data in $\mathbb{R}^2$}

\author[]{Animesh Jana}
\affil[]{\footnotesize Harish-Chandra Research Institute, A CI of Homi Bhabha National Institute,\\ Chhatnag Road, Jhunsi, Prayagraj 211019, India.}
\affil[]{\em \footnotesize	animeshjana@hri.res.in , ajana.math@gmail.com}

\begin{document}
	\maketitle
	\begin{abstract}
		In this note, we give an explicit construction of global solutions to the prescribed Jacobian equation 
		\[
		\det(\nabla u)=f\mbox{ in }\mathbb{R}^2,
		\]
		for a class of data. For every $f\in C_c^1(\mathbb{R}^2)$ and $p>1$, we construct a solution $u\in \dot W^{1,2p}(\mathbb{R}^2)\cap L^\infty(\mathbb{R}^2)$. In particular, no sign condition or integral constraint is imposed on the compactly supported data. For $f\in C_c^\infty(\mathbb{R}^2)$, the construction yields a smooth solution. We also consider rapidly decaying data and prove the existence of global solutions with bounded gradient for every $f\in\mathcal{S}(\mathbb{R}^2)$. Finally, we extend the construction to compactly supported data that are measurable in one variable and $C^1$ in the other. Our proof relies on a similar idea as in [Moser, Trans. Amer. Math. Soc. 1965].
	\end{abstract}	

\section{Introduction}
In this short note, we study the prescribed Jacobian equation
\begin{equation}\label{eq:intro-jacobian}
	\det(\nabla u)=f \qquad \text{in }\mathbb{R}^n, 
\end{equation}
with particular emphasis on the two-dimensional case \(n=2\). Here $f$ is given and it belongs to a suitable function space. We want to find the unknown $u:\mathbb{R}^n\to\mathbb{R}^n$ in a homogeneuous Sobolev spaces. Our main objective is to construct global solutions for compactly supported data without imposing either a sign condition on \(f\) or an integral constraint on \(f\).

The prescribed Jacobian equation has a long history and is closely connected with nonlinear analysis, geometric mapping theory, and compensated compactness. Our focus is for the problem in whole of $\R^n$, in particular $n=2$. This problem has been initiated by Coifman, Lions, Meyer and Semmes \cite{Coifman-et-al}. We define $\mathcal{J}u=\mbox{det}(\nabla u)$. It still remains an open problem in this topic:
\begin{question}\label{Q1}
	Is the Jacobian $\mathcal{J}: \dot{W}^{1,np}(\R^n,\R^n)\rr \mathscr{H}^p(\R^n)$ surjective?
\end{question}
Here $\mathscr{H}^p(\R^n)$ denotes the Hardy space in $\R^n$. It has been shown \cite{Coifman-et-al} that $\mathscr{H}^p(\R^n)$ is the smallest Banach space that includes the range of the Jacobian map $\mathcal{J}$. Moreover, if we replace the domain of $\mathcal{J}$ by inhomogeneous Sobolev space, then there exists counter examples (see \cite{LindbergARMA17}).

The existence of solutions to the Jacobian equation $\det(\nabla u) = f$ is well-studied, most notably through the work of Moser \cite{Moser65} and Dacorogna--Moser \cite{DacorognaMoser90}. However, classical methods generally require $f$ to satisfy an integral constraint to ensure $u$ is the identity at infinity. In this paper, we provide an entirely elementary construction that proves for any $f \in C_c^\infty(\mathbb{R}^2)$, there exists a smooth solution $u$. Our construction is explicit: $u_1$ remains constant at infinity while $u_2$ is supported within a strip. 

In a bounded, smooth domain $\Omega\subset\R^2$, the Jacobian equation is usually treated as a Dirichlet problem: finding $u: \Omega \to \R^2$ such that $\det(\nabla u )= f$ with $u(x) = x$ (or some other prescribed map) on $\partial \Omega$. In the seminal paper \cite{DacorognaMoser90}, Dacoronga and Moser studied the existence of solution for Jacobian equation in a domain when $f>0$. A necessary condition for existence is the total volume constraint. By the change of variables formula, one must have the following integral condition
\begin{equation*}
	\int_{\Omega} f(x) \, dx = \text{vol}(\Omega).	
\end{equation*} 
Dacorogna and Moser \cite{DacorognaMoser90} proved that if $f > 0$ and this integral condition holds, smooth solutions exist. Li-Schikorra \cite{LiSchikorra} obtained the existence result when $f$ belongs to fractional Sobolev space and $f>0$. See also \cite{CupiniDacorognaKneuss,Kneuss2012,Ye94} and references therein. Recently, Guerra, Koch, and Lindberg \cite{GuerraKochLindbergCVPDE21} showed that the in critical Sobolev spaces, even if $f$ is close to 1, a solution $u$ in the expected Sobolev class $W^{1,n}$ may not exist for a Baire-generic $f$. This highlights that the boundary constraint $u|_{\partial \Omega} = \text{id}$ is very restrictive when combined with the nonlinear determinant. For bounded domains, Guerra, Koch, and Lindberg \cite{GuerraKochLindbergARMA21} study the equation through the lens of energy minimization. They seek to find $u$ that minimizes the following functional 
\begin{equation*}
	\int_{\Omega} |\nabla u|^2\,dx\mbox{ subject to }\det(\nabla u)= f. 
\end{equation*}
For bounded domain, the regularity aspects of solution to the Jacobian equation has been studied in \cite{BrezisNguyen,Muller}

In $\mathbb{R}^n$, the focus shifts from boundary values to integrability, decay, and harmonic analysis structures. The landmark result by Coifman, Lions, Meyer, and Semmes \cite{Coifman-et-al} established that for any $u$ such that $\nabla u \in L^n(\mathbb{R}^n)$, the Jacobian $\det(\nabla u)$ lies in the Hardy space $\mathscr{H}^1(\mathbb{R}^n)$. Iwaniec \cite{Iwaniec} asked about the surjectivity of the Jacobian map. It has been shown in \cite{GuerraKochLindberg23} that the surjectivity is equivalent to the continuity of the inverse and also to the estimate $\norm{\nabla u}_{L^{np}(\R^n)}^n\lesssim \norm{\mathcal{J}u}_{\mathscr{H}^p(\R^n)}$. Hyt\"{o}nen \cite{Hytonen} and Lindberg \cite{LindbergJFAA23,LindbergThesis} focused on the surjective nature of the Jacobian in $\mathbb{R}^n$.  For radially symmetric data, partial answer to Question \ref{Q1} can be found here \cite{LindbergThesis}. Lindberg proved that the existence of radial stretching solutions for radial data $f$ satisfying the following condition,
\begin{equation*}
	\int_{B(0,r)}f(|x|)dx\mbox{ does not change sign for all }r>0.
\end{equation*}
Furthermore, in \cite{GuerraKochLindbergARMA21}, Guerra et al showed that for a large class of data, the uniqueness up to rotation can be characterized by $\la[f]$ which is defined as follows
\begin{equation*}
	\la[f]=\mbox{ess--}\hspace{-0.2cm}\sup\limits_{r\in[0,1]}\frac{|f(r)|}{\frac{1}{|B(0,r)|}\int_{B(0,r)}f(x)\,dx}.
\end{equation*}
They have also shown the existence and uniqueness of energy minimizer solution for radially symmetric data satisfying $\la[f]\leq 1$. Later, it has been shown for a larger class of data by Lindberg \cite{LindbergJFAA23}. The results presented in the current paper address a different question from those considered in these works. 

The purpose of this note is more elementary. We exploit the special structure of the determinant in two dimension to construct solutions explicitly. Our main result shows, in particular, that every smooth compactly supported function on \(\mathbb{R}^2\), with no restriction on its sign or total integral, can be realized as the Jacobian determinant of a globally defined bounded smooth map. More precisely, for every $ f\in C_c^\infty(\mathbb{R}^2)$ and every \(p>1\), we construct a map \(u:\mathbb{R}^2\to\mathbb{R}^2\) satisfying
\begin{equation*}
	\det(\nabla u)=f,
\end{equation*}
and having the Sobolev and boundedness properties (see Theorem \ref{theorem-1}). The construction is explicit and does not rely on a continuity or approximation argument.

The absence of an integral constraint is an essential feature of the construction. Indeed, if one additionally required \(u\) to have compact support, then under suitable regularity assumptions the integral of its Jacobian over \(\mathbb{R}^2\) would vanish. Thus arbitrary compactly supported data with nonzero integral cannot arise from compactly supported solutions. Our construction instead allows one component of the solution to approach a nonzero function at infinity, while the other component remains localized. This makes it possible to realize data with arbitrary total integral.

We also give a related construction for rapidly decaying data. In this case, a simple separation of variables allows one to construct a global solution with bounded gradient. Finally, we consider discontinuous data and give explicit examples for characteristic functions of convex sets. These examples illustrate that the basic mechanism is not restricted to smooth data.

In this short note, exploring the idea in Moser's construction \cite{Moser65}, we are going to show the existence of solution in $\R^2$ for smooth compact support data. 

\begin{theorem}\label{theorem-1}
	Let $f:\R^2\rr\R$ be a compactly supported function such that $x\mapsto f(x,y)$ is a bounded measurable for all $y\in\R$ and $y\mapsto f(x,y)$ is $C^1$ for all $x\in\R$. For each $p>1$, there exists a function $u\in \dot{W}^{1,2p}(\R^2)\cap L^\f(\R^2)$ such that
	\begin{equation}
		\det(\nabla u(x,y))= f(x,y) \mbox{ for a.e. }(x,y)\in \R^2.
	\end{equation}
	
\end{theorem} 
We would like to mention that the solution constructed in Theorem \ref{theorem-1} for radially symmetric data does not satisfy the {\em radial stretching} condition as defined in Lindberg \cite{LindbergThesis}. The construction works for all $C_c^1$ functions. Although, main interest of the topic is to find radial stretching solution for a radially symmetric data, our construction only guarantees the existence but no symmetry. 
\begin{corollary}
	Let $f\in C_c^\f(\R^2)$. For each $p>1$, there exists a function $u\in \dot{W}^{1,2p}(\R^2)\cap L^\f(\R^2)$ such that
	\begin{equation}
		\det(\nabla u(x,y))= f(x,y) \mbox{ for all }(x,y)\in \R^2.
	\end{equation}
	Furthermore, it holds that $u\in C^\f(\R^2)$.
\end{corollary}
The key highlights of our results are following.
\begin{itemize}
	\item {\em No integral constraint:} By proving this for any $f \in C_c^\infty$, we have bypassed the restrictive barrier in the Jacobian equation.
	\item {\em Specific Geometry:} In the Dacorogna-Moser flow, the solution typically spreads out in all directions. Our construction is anisotropic (one component is constant at infinity, the other is localized in a strip).
\end{itemize}

The note is organized as follows. In Section \ref{sec:proof-main} we give the explicit constructions for smooth compactly supported and rapidly decaying data. Section \ref{sec:discont} contains examples for discontinuous data, including characteristic functions of convex compact sets.

\section{Proof of main result}\label{sec:proof-main}
In this section, we present our construction of solution $u$ to the Jacobian equation \eqref{eq:intro-jacobian} for $C_c^\f$ data $f$.  One observation that we can make is the following: suppose $u_1$ and $u_2$ are defined such a way that $\nabla u_1=(f(x,y),0)$ and $\nabla u_2=(0,1)$ in the support of $f$. Then we can check that this simple choice gives us the solution. Difficulty arises when we try to find such $u_1$ and $u_2$. Since outside of the support of $f$ we need the determinant of $\nabla u$ to be $0$, it requires the $\nabla u_1$ and $\nabla u_2$ to be parallel. In our construction below, we are going to exploit such idea by considering $u_1(x,y)$ in terms of $u_2$. The proof follows a similar approach as in Moser \cite{Moser65} except the choice of $u_2$.
%
%
%

\begin{proof}[Proof of Theorem \ref{theorem-1}]
		Let $R_0>0$ be such that $\mbox{supp}(f)\subset (-R_0,R_0)\times (-R_0,R_0)$. We wish to construct functions $u_i:\R^2\rr\R$ for $i=1,2$ such that $\det(\nabla u)=f$ where $u=(u_1,u_2)$. Let us consider $u_2(x,y)$ to be a Lipschitz function on $\R^2$ such that
	\begin{enumerate}[(i.)]
		\item For $|x|\leq 2R_0$ and $|y|\geq 2R_0$, $|u_2(x,y)|\geq R_0$.
		\item For $|x|\leq 2R_0$ and $|y|\leq 2R_0$, $u_2(x,y)=y$. 
	\end{enumerate}
	We will give an explicit construction of such $u_2$ later. Suppose that such $u_2$ exists. Similar to \cite{Moser65}, we then set
	\begin{equation*}
		u_1(x,y):=	\int_{-R_0}^{x}f(t,u_2(x,y))dt.
	\end{equation*}
	From our choice of $u_2$, we get
	\begin{equation*}
		u_1(x,y)=	\int_{-R_0}^{x}f(t,R_0)dt=0\mbox{ when }|x|\leq R_0\mbox{ and } |y|\geq R_0,
	\end{equation*}
	since $f(t,z)=0$ for any $t\in\R$ and $|z|\geq R_0$. We can now calculate
	\begin{align}
		\pa_x u_1(x,y) &=f(x,u_2(x,y))+\pa_xu_2(x,y)\int_{-R_0}^{x}\pa_2f(t,u_2(x,y))dt,\label{derivative-u1-x}\\
		\pa_y u_1(x,y) &=\pa_yu_2(x,y)\int_{-R_0}^{x}\pa_2f(t,u_2(x,y))dt.\label{derivative-u1-y}
	\end{align}
	Here $\pa_2$ denotes the partial derivative of the function $f$ with respect to second argument. Now, we may calculate $\det(\nabla u)$ as follows
	\begin{align*}
		\det(\nabla u(x,y))&=\pa_x u_1(x,y)\pa_y u_2(x,y)-\pa_y u_1(x,y)\pa_x u_2(x,y)\\
		&=f(x,u_2(x,y))\pa_y u_2(x,y)+\pa_xu_2(x,y)\pa_y u_2(x,y)\int_{-R_0}^{x}\pa_2f(t,u_2(x,y))dt\\
		&-\pa_x u_2(x,y)\pa_yu_2(x,y)\int_{-R_0}^{x}\pa_2f(t,u_2(x,y))dt\\
		&=f(x,u_2(x,y))\pa_y u_2(x,y).
	\end{align*}
	Observe that the following
	\begin{enumerate}
		\item when $|x|\geq R_0$ we have $f(x,u_2(x,y))=0$.
		\item For $|x|\leq R_0$ and $|y|\geq R_0$ from property (i) of $u_2$ we again get $f(x,u_2(x,y))=0$ since support of $f$ is contained in side $(-R_0,R_0)\times (-R_0,R_0)$. 
		\item When $|x|\leq R_0$ and $|y|\leq R_0$ , we have $\pa_y u_2(x,y)=1$ from property (ii).
	\end{enumerate} 
	Hence, we have
	\begin{equation}
		\det(\nabla u(x,y))=f(x,y)\mbox{ for all }(x,y)\in\R^2. 
	\end{equation}
	We define the now, the function $u_2$ as follows
	\begin{enumerate}
		\item For $-\f<x\leq -R_0$,
		\begin{equation*}
			u_2(x,y)=\left\{
			\begin{array}{cl}
				\frac{2R_0^2y}{R_0^2+x^2}&\text{ for }-\f<x\leq -R_0,|y|\leq R_0,\\
				\frac{2R_0^3\left(3R_0-x-y\right)}{(R_0^2+x^2)(2R_0-x)}+\frac{R_0(y-R_0)}{2R_0-x}&\text{ for }-\f<x\leq -R_0, R_0\leq y\leq 3R_0-x,\\
				R_0&\text{ for }-\f<x\leq -R_0, y\geq 3R_0-x,\\
				-\frac{2R_0^3\left(3R_0-x+y\right)}{(R_0^2+x^2)(2R_0-x)}+\frac{R_0(y+R_0)}{2R_0-x}&\text{ for }-\f<x\leq -R_0, -3R_0+x\leq y\leq -R_0,\\
				-R_0&\text{ for }-\f<x\leq -R_0, y\leq -3R_0+x.
			\end{array}\right.
		\end{equation*}
		\item For $-R_0\leq x\leq R_0$,
		\begin{equation*}
			u_2(x,y)=\left\{
			\begin{array}{cl}
				y&\text{ for }-R_0\leq x\leq R_0,|y|\leq R_0,\\
				R_0&\text{ for }-R_0\leq x\leq -R_0, R_0\leq y<+\f,\\ 
				-R_0&\text{ for }-R_0\leq x\leq -R_0, -\f<y\leq -R_0.
			\end{array}\right.
		\end{equation*}
		\item For $R_0\leq x<+\f$,
		\begin{equation*}
			u_2(x,y)=\left\{
			\begin{array}{cl}
				\frac{2R_0^2y}{R_0^2+x^2}&\text{ for }R_0\leq x< +\f,|y|\leq R_0,\\
				\frac{2R_0^3\left(3R_0+x-y\right)}{(R_0^2+x^2)(2R_0+x)}+\frac{R_0(y-R_0)}{2R_0+x}&\text{ for }R_0\leq x<+\f, R_0\leq y\leq 3R_0+x,\\
				R_0&\text{ for }R_0\leq x<+\f, y\geq 3R_0+x,\\
				-\frac{2R_0^3\left(3R_0+x+y\right)}{(R_0^2+x^2)(2R_0+x)}+\frac{R_0(y+R_0)}{2R_0+x}&\text{ for }R_0\leq x<+\f, -3R_0-x\leq y\leq -R_0,\\
				-R_0&\text{ for }R_0\leq x<+\f, y\leq -3R_0-x.
			\end{array}\right.
		\end{equation*}
	\end{enumerate}
	One can check that $u_2\in Lip(\R^2)$. Next we show that $\nabla u\in \dot{W}^{1,2p}(\R^2)$. From \eqref{derivative-u1-x} and \eqref{derivative-u1-y}, we can estimate $||\cdot||_{L^{2p}(\R^2)}$ of $\nabla u_1$ as follows.
	\begin{align*}
		&\int_{\R^2}(|\pa_xu_1(x,y)|^{2p}+|\pa_yu_1(x,y)|^{2p})dxdy\\
		&\leq C_p\int_{\R}\int_{-R_0}^{R_0}|f(x,u_2(x,y))|^{2p}dxdy+C_p\int_{-R_0}^{\f}\int_{\R}|\pa_xu_2(x,y)|^{2p}\int_{-R_0}^{x}|\pa_2f(t,u_2(x,y))|^{2p}dtdydx\\
		&+\int_{-R_0}^{\f}\int_{\R}|\pa_yu_2(x,y)|^{2p}\int_{-R_0}^{x}|\pa_2f(t,u_2(x,y))|^{2p}dtdydx,
	\end{align*}
	for some constant $C_p>1$. Here we have used Jensen's inequality. Therefore, we get
	\begin{align*}
		&\int_{\R^2}(|\pa_xu_1(x,y)|^{2p}+|\pa_yu_1(x,y)|^{2p})dxdy\\
		&\leq C_p\int_{-R_0}^{R_0}\int_{-R_0}^{R_0}|f(x,u_2(x,y))|^{2p}dxdy\\
		&+C_p\left(\sup\limits_{z\in \R}\int_{-R_0}^{R_0}|\pa_2f(t,z)|^{2p}dt\right)\left(\int_{-R_0}^{\f}\int_{\R}|\pa_xu_2(x,y)|^{2p}dydx+\int_{-R_0}^{\f}\int_{\R}|\pa_yu_2(x,y)|^{2p}dydx\right)\\
		&\leq C_p\int_{-R_0}^{R_0}\int_{-R_0}^{R_0}|f(x,u_2(x,y))|^{2p}dxdy+C_p\left(\sup\limits_{z\in \R}\int_{-R_0}^{R_0}|\pa_2f(t,z)|^{2p}dt\right)||\nabla u_2||^{2p}_{L^{2p}(\R^2)}.
	\end{align*}
Now, it remains to prove that $\nabla u_2\in L^{2p}(\R^2)$. We can obtain
\begin{align*}
			&\int_{\R^2}(|\pa_xu_2(x,y)|^{2p}+|\pa_yu_2(x,y)|^{2p})dxdy\\
			&\leq \int_{\R}\int_{-\f}^{-R_0}(|\pa_xu_2(x,y)|^{2p}+|\pa_yu_2(x,y)|^{2p})dxdy\\
			&+\int_{\R}\int_{-R_0}^{R_0}(|\pa_xu_2(x,y)|^{2p}+|\pa_yu_2(x,y)|^{2p})dxdy\\
			&+\int_{\R}\int_{R_0}^{\f}(|\pa_xu_2(x,y)|^{2p}+|\pa_yu_2(x,y)|^{2p})dxdy\\
			&=:I_1+I_2+I_3.
\end{align*}
Since $\nabla u_2\in L^\f(\R^2)$, we can conclude that $I_2<+\f$. Next we are going to show that $I_1<+\f$. Proof for $I_3<+\f$ follows in a similar way and we will omit here. From the definition, we can calculate the derivative of $u_2$ as follows.
\begin{enumerate}
	\item For $-\f<x< -R_0,|y|< R_0$, we have
	\begin{equation*}
		\pa_x u_2(x,y)=-\frac{4R_0^2xy}{(R^2_0+x^2)^2}\mbox{ and }\pa_y u_2(x,y)=\frac{2R_0^2}{R_0^2+x^2}.
	\end{equation*}
	\item For $-\f<x< -R_0, R_0< y<3R_0-x$, it follows
	\begin{align*}
		\pa_x u_2(x,y)&=-\frac{2R_0^3}{(R_0^2+x^2)(2R_0-x)}+\frac{2R_0^3\left(3R_0-x-y\right)}{(R_0^2+x^2)(2R_0-x)^2}\\
		&-\frac{4R_0^3x\left(3R_0-x-y\right)}{(R_0^2+x^2)^2(2R_0-x)}+\frac{R_0(y-R_0)}{(2R_0-x)^2},\\
		\pa_y u_2(x,y)&=-\frac{2R_0^3}{(R_0^2+x^2)(2R_0-x)}+\frac{R_0}{2R_0-x}.
	\end{align*}
	\item For $-\f<x< -R_0, y> 3R_0-x$, we get
	\begin{equation*}
		\pa_xu_2(x,y)=0\mbox{ and }\pa_yu_2(x,y)=0.
	\end{equation*}
	\item For $-\f<x< -R_0, -3R_0+x< y<-R_0$, it follows
	\begin{align*}
			\pa_x u_2(x,y)&=\frac{2R_0^3}{(R_0^2+x^2)(2R_0-x)}-\frac{2R_0^3\left(3R_0-x+y\right)}{(R_0^2+x^2)(2R_0-x)^2}\\
			&+\frac{4R_0^3x\left(3R_0-x+y\right)}{(R_0^2+x^2)^2(2R_0-x)}+\frac{R_0(y+R_0)}{(2R_0-x)^2},\\
			\pa_y u_2(x,y)&=-\frac{2R_0^3}{(R_0^2+x^2)(2R_0-x)}+\frac{R_0}{2R_0-x}.
	\end{align*}
	\item For $-\f<x\leq -R_0, y< -3R_0+x$, we have
	\begin{equation*}
		\pa_xu_2(x,y)=0\mbox{ and }\pa_yu_2(x,y)=0.
	\end{equation*}
\end{enumerate}
Then, it follows
\begin{align*}
	&\int_{-\f}^{-R_0}\int_\R (|\pa_xu_1(x,y)|^{2p}+|\pa_yu_1(x,y)|^{2p})dxdy\\
	&\leq 	C_p\int_{-\f}^{-R_0}\int_{-R_0}^{R_0}\left(\frac{(4R_0^2)^{2p}|x|^{2p}|y|^{2p}}{(R^2_0+x^2)^{4p}}+\frac{(2R_0^2)^{2p}}{(R_0^2+x^2)^{2p}}\right)dydx\\
	&+C_p\int_{-\f}^{-R_0}\int_{R_0}^{3R_0-x}\left(\frac{(2R_0^3)^{2p}}{(R_0^2+x^2)^{2p}(2R_0-x)^{2p}}+\frac{(2R_0^3)^{2p}|3R_0-x-y|^{2p}}{(R_0^2+x^2)^{2p}(2R_0-x)^{4p}}\right)dydx\\
	&+C_p\int_{-\f}^{-R_0}\int_{R_0}^{3R_0-x}\left(\frac{(4R_0^3)^{2p}|x|^{2p}|3R_0-x-y|^{2p}}{(R_0^2+x^2)^{4p}(2R_0-x)^{2p}}+\frac{R_0^{2p}|y-R_0|^{2p}}{(2R_0-x)^{4p}}\right)dydx\\
	&+C_p\int_{-\f}^{-R_0}\int_{R_0}^{3R_0-x}\left(\frac{(2R_0^3)^{2p}}{(R_0^2+x^2)^{2p}(2R_0-x)^{2p}}+\frac{R_0^{2p}}{(2R_0-x)^{2p}}\right)dydx\\
	&+C_p\int_{-\f}^{-R_0}\int_{-3R_0+x}^{-R_0}\left(\frac{(2R_0^3)^{2p}}{(R_0^2+x^2)^{2p}(2R_0-x)^{2p}}+\frac{(2R_0^3)^{2p}|3R_0-x+y|^{2p}}{(R_0^2+x^2)^{2p}(2R_0-x)^{4p}}\right)dydx\\
	&+C_p\int_{-\f}^{-R_0}\int_{-3R_0+x}^{-R_0}\left(\frac{(4R_0^3)^{2p}|x|^{2p}|3R_0-x+y|^{2p}}{(R_0^2+x^2)^{4p}(2R_0-x)^{2p}}+\frac{R_0^{2p}|y+R_0|^{2p}}{(2R_0-x)^{4p}}\right)dydx\\
	&+C_p\int_{-\f}^{-R_0}\int_{-3R_0+x}^{-R_0}\left(\frac{(2R_0^3)^{2p}}{(R_0^2+x^2)^{2p}(2R_0-x)^{2p}}+\frac{R_0^{2p}}{(2R_0-x)^{2p}}\right)dydx\\
	&=I_{11}+I_{12}+I_{13}+I_{14}+I_{15}+I_{16}+I_{17}.
\end{align*}
Next we estimate $I_{11},I_{12},I_{13}$ and $I_{14}$. 
\begin{align*}
		I_{11}&=C_p\int_{-\f}^{-R_0}\int_{-R_0}^{R_0}\left(\frac{(4R_0^2)^{2p}|x|^{2p}|y|^{2p}}{(R^2_0+x^2)^{4p}}+\frac{(2R_0^2)^{2p}}{(R_0^2+x^2)^{2p}}\right)dydx\\
				&\leq 	C_p2^{4p}R_0^{4p}\int_{-\f}^{-R_0}\int_{-R_0}^{R_0}\left(\frac{|x|^{2p}|y|^{2p}}{(R^2_0+x^2)^{4p}}+\frac{1}{(R_0^2+x^2)^{2p}}\right)dydx\\
				&\leq 	C_p2^{4p+1}R_0^{4p+1}\int_{-\f}^{-R_0}\frac{2}{(R^2_0+x^2)^{2p}}dx\\
				&\leq 	C_p2^{4p+1}R_0^{4p+1}\int_{-\f}^{-R_0}\frac{2^{2p+1}}{(R_0-x)^{2p}}dx\quad(\mbox{since }2(R^2_0+x^2)\geq (R_0-x)^2)\\
				&=\frac{C_p}{2p-1}2^{6p+2}R_0^{2p+2}.
\end{align*}
Then, we analyze $I_{12}$.
\begin{align*}
	I_{12}&=C_p\int_{-\f}^{-R_0}\int_{R_0}^{3R_0-x}\left(\frac{(2R_0^3)^{2p}}{(R_0^2+x^2)^{2p}(2R_0-x)^{2p}}+\frac{(2R_0^3)^{2p}|3R_0-x-y|^{2p}}{(R_0^2+x^2)^{2p}(2R_0-x)^{4p}}\right)dydx\\
	&=C_p\int_{-\f}^{-R_0}\left(\frac{(2R_0^3)^{2p}(2R_0-x)}{(R_0^2+x^2)^{2p}(2R_0-x)^{2p}}+\frac{1}{2p+1}\frac{(2R_0^3)^{2p}|2R_0-x|^{2p+1}}{(R_0^2+x^2)^{2p}(2R_0-x)^{4p}}\right)dx\\
	&\leq C_p(2R_0^3)^{2p}\int_{-\f}^{-R_0}\frac{2}{(R_0^2+x^2)^{2p}(2R_0-x)^{2p-1}}dx\\
	&\leq C_p(2R_0^3)^{2p}\int_{-\f}^{-R_0}\frac{2^{2p+1}}{(R_0-x)^{2p}}\frac{1}{(3R_0)^{2p-1}}dx\quad(\mbox{since }2(R^2_0+x^2)\geq (R_0-x)^2)\\
	&=\frac{C_p(2R_0^3)^{2p}\cdot 2^{2p+1}}{(2p-1)(2R_0)^{2p-1}(3R_0)^{2p-1}}.
\end{align*}
Now, we estimate $I_{13}$ as follows.
\begin{align*}
	I_{13}&=C_p\int_{-\f}^{-R_0}\int_{R_0}^{3R_0-x}\left(\frac{(4R_0^3)^{2p}|x|^{2p}|3R_0-x-y|^{2p}}{(R_0^2+x^2)^{4p}(2R_0-x)^{2p}}+\frac{R_0^{2p}|y-R_0|^{2p}}{(2R_0-x)^{4p}}\right)dydx\\
	&=\frac{C_p}{2p+1}\int_{-\f}^{-R_0}\left(\frac{(4R_0^3)^{2p}|x|^{2p}(2R_0-x)^{2p+1}}{(R_0^2+x^2)^{4p}(2R_0-x)^{2p}}+\frac{R_0^{2p}(2R_0-x)^{2p+1}}{(2R_0-x)^{4p}}\right)dx\\
&=\frac{C_p}{2p+1}\int_{-\f}^{-R_0}\left(\frac{(4R_0^3)^{2p}|x|^{2p}(2R_0-x)}{(R_0^2+x^2)^{4p}}+\frac{R_0^{2p}}{(2R_0-x)^{2p-1}}\right)dx\\
&\leq \frac{C_p}{2p+1}\int_{-\f}^{-R_0}\left(\frac{2(4R_0^3)^{2p}(R_0-x)}{(R_0^2+x^2)^{2p}}+\frac{R_0^{2p}}{(2R_0-x)^{2p-1}}\right)dx\\
&\leq \frac{C_p}{2p+1}\int_{-\f}^{-R_0}\left(\frac{2^{2p+1}(4R_0^3)^{2p}}{(R_0-x)^{4p-1}}+\frac{R_0^{2p}}{(2R_0-x)^{2p-1}}\right)dx\quad(\mbox{since }2(R^2_0+x^2)\geq (R_0-x)^2)\\
&= \frac{C_p}{2p+1} \left(\frac{1}{4p-2}\frac{2^{2p+1}(4R_0^3)^{2p}}{(2R_0)^{4p-2}}+\frac{1}{2p-2}\frac{R_0^{2p}}{(3R_0)^{2p-2}}\right).
\end{align*}
Finally, we bound $I_{14}$ as follows.
\begin{align*}
	I_{14}&=C_p\int_{-\f}^{-R_0}\int_{R_0}^{3R_0-x}\left(\frac{(2R_0^3)^{2p}}{(R_0^2+x^2)^{2p}(2R_0-x)^{2p}}+\frac{R_0^{2p}}{(2R_0-x)^{2p}}\right)dydx\\
	&=C_p\int_{-\f}^{-R_0}\left(\frac{(2R_0^3)^{2p}(2R_0-x)}{(R_0^2+x^2)^{2p}(2R_0-x)^{2p}}+\frac{R_0^{2p}(2R_0-x)}{(2R_0-x)^{2p}}\right)dx\\
		&=C_p\int_{-\f}^{-R_0}\left(\frac{(2R_0^3)^{2p}}{(R_0^2+x^2)^{2p}(2R_0-x)^{2p-1}}+\frac{R_0^{2p}}{(2R_0-x)^{2p-1}}\right)dx\\
		&\leq C_p\int_{-\f}^{-R_0}\left(\frac{2^{2p}(2R_0^3)^{2p}}{(R_0-x)^{2p}(3R_0)^{2p-1}}+\frac{R_0^{2p}}{(2R_0-x)^{2p-1}}\right)dx\\
		&=C_p\left(\frac{2^{2p}(2R_0^3)^{2p}}{(2p-1)(2R_0)^{2p-1}(3R_0)^{2p-1}}+\frac{R_0^{2p}}{(2p-2)(3R_0)^{2p-2}}\right).
\end{align*}
Note that here we have used the fact that $p>1$. The proof of bounds for $I_{15},I_{16}$ and $I_{17}$ follows by similar arguments as for $I_{12}$, $I_{13}$ and $I_{14}$ respectively. We omit the details here. This completes the proof of Theorem \ref{theorem-1}.
\end{proof}
\begin{remark}
	\begin{enumerate}
		\item 	We remark that solutions constructed in the proof of Theorem \ref{theorem-1} depend on the choice of cut-off parameters. It is needless to say that these solutions are not unique.
		\item It is not difficult to see that we can construct a $C^\f$ solution $u$ from the construction above by considering the smooth $u_2$ instead of Lipschitz one.
		\item As it has been pointed out in the proof that our construction only provides $\nabla u\in L^{2p}$ for $p>1$. We can not repeat the same for $p=1$.
	\end{enumerate}

\end{remark}

	\subsection{Construction of solutions for Schwartz functions}\label{sec:Schwartz}
	By exploring the idea from proof of Theorem \ref{theorem-1}, we can also construct the solution for rapidly decaying data.
	\begin{proposition}\label{prop:Schwartz}
		For any $h\in \mathcal{S}(\R^2)$ there exists $u=(u_1,u_2)\in C^1(\R^2)$ such that
		\begin{equation}
			\det(\nabla u(x,y))=h(x,y)\mbox{ for all }(x,y)\in \R^2.
		\end{equation}
		Furthermore, it follows $\nabla u\in L^{\f}(\R^2)$.
	\end{proposition}
	\begin{proof}
		Now we consider $u=(u_1,u_2)$ as follows
		\begin{equation}
			u_1(x,y)=(1+|y|^\B)\int_0^xh\left(t,y\right)\,dt\mbox{ and }u_2(x,y)=\int_0^y\frac{1}{1+|s|^\B}\,ds,
		\end{equation}
		for some integer $\B\geq2$. It can be computed that
		\begin{equation*}
			\nabla u(x,y)=\begin{pmatrix}
				(1+|y|^\B)h(x,y)&(1+|y|^\B)\int_0^x\pa_yh(t,y)\,dt+\B sgn(y)|y|^{\B-1}\int_0^xh(t,y)\,dt\\
				0&\frac{1}{1+|y|^\B}
			\end{pmatrix}.
		\end{equation*}
		We can check that $	\det(\nabla u(x,y))=h(x,y)$ and $\nabla u\in L^\f(\R^2)$. This completes the proof of Proposition \ref{prop:Schwartz}.
	\end{proof}

\section{Example of solutions for discontinuous data}\label{sec:discont}
In this section, we construct solutions corresponding to a specific class of discontinuous data. More specifically, we consider data that are characteristic function of a convex set. For $1\leq p<+\f$, we denote $|(x_1,x_2)|_p=(|x_1|^p+|x_2|^p)^{1/p}$.
\begin{example}
	Let $f$ be the indicator function of unit ball under $\abs{\cdot}_{1}$, that is, $f=\mathbbm{1}_{B_1}$ where $B_1=\{x;\,\abs{x}_1=|x_1|+|x_2|\leq 1\}$. We set
	\begin{equation}
		u^+(x):=\left\{\begin{array}{cl}
			x&\mbox{ for }\abs{x}_1\leq 1,\\
			\frac{x}{|x_1|+|x_2|} &\mbox{ for }\abs{x}_1>1.
		\end{array}\right.
	\end{equation}
	Subsequently, it follows from an elementary computation, 
	\begin{equation}
		\nabla u^+(x)=\left\{\begin{array}{cl}
			I_2&\mbox{ for }\abs{x}_1< 1,\\
			\frac{1}{|x|^3_1} \begin{pmatrix}
				|x_2|&-\mbox{sgn}(x_2)x_1\\
				-\mbox{sgn}(x_1)x_2&|x_1|
			\end{pmatrix}&\mbox{ for }\abs{x}_1>1.
		\end{array}\right.
	\end{equation}
	Then it follows,
	\begin{equation}
		\det(\nabla u^+(x))=\left\{\begin{array}{cl}
			1&\mbox{ for }\abs{x}_1< 1,\\
			0&\mbox{ for }\abs{x}_1>1.
		\end{array}\right.
	\end{equation}
	We can also consider
	\begin{equation}
		u^-(x):=\left\{\begin{array}{cl}
			(x_2,x_1)&\mbox{ for }\abs{x}_1\leq 1,\\
			\frac{(x_2,x_1)}{|x_1|+|x_2|} &\mbox{ for }\abs{x}_1>1.
		\end{array}\right.
	\end{equation}
	Then, we get
	\begin{equation}
		\nabla u^-(x)=\left\{\begin{array}{cl}
			\begin{pmatrix}
				0&1\\
				1&0
			\end{pmatrix}&\mbox{ for }\abs{x}_1< 1,\\
			\frac{1}{|x|^3_1} \begin{pmatrix}
				|x_1|&-\mbox{sgn}(x_1)x_2\\
				-\mbox{sgn}(x_2)x_1&|x_2|
			\end{pmatrix}&\mbox{ for }\abs{x}_1>1.
		\end{array}\right.
	\end{equation}
	Consequently, we obtain
	\begin{equation}
		\det(\nabla u^-(x))=\left\{\begin{array}{cl}
			-1&\mbox{ for }\abs{x}_1< 1,\\
			0&\mbox{ for }\abs{x}_1>1.
		\end{array}\right.
	\end{equation}
\end{example}

\begin{proposition}\label{prop:discont}
	Let $K\subset\R^d$ with $d\geq2$ be a convex compact set. Then there exists $u\in \dot{W}^{1,dp}(\R^d)\cap L^\f(\R^d)$ such that
	\begin{equation}
		\det(\nabla u)=\mathbbm{1}_{K}.
	\end{equation}
\end{proposition}

\begin{proof}
	Define Minkowski functional $\mu_k:\R^d\rr[0,\f)$ as follows
	\begin{equation}
		\mu_K(x)=\inf\left\{t>0;\,t^{-1}x\in K\right\}.
	\end{equation}
	Observe that for any $x\in K$ we have $\mu_K(x)\leq1$ and when $x\in K^c$ it follows $\mu_K(x)>1$. Furthermore, Lipschitz continuity of the map $x\mapsto\mu_K(x)$ follows. We note that $\{x\in\R^d;\,\mu_K(x)\leq1\}=K$. Consider $u:\R^d\rr\R^d$ defined as follows
	\begin{equation}
		u^+(x)=\left\{\begin{array}{cl}
			x&\mbox{ for }\mu_K(x)<1,\\
			\frac{x}{\mu_K(x)}&\mbox{ for }\mu_K(x)>1.
		\end{array}\right.
	\end{equation}
	By a simple calculation, it follows
	\begin{equation}
		\nabla u^+(x)=\left\{\begin{array}{cl}
			I_d&\mbox{ for }\mu_K(x)\leq1,\\
			\frac{1}{\mu_K(x)}I_d-\frac{x\otimes \nabla\mu_K(x)}{\mu^2_K(x)}&\mbox{ for }\mu_K(x)>1.
		\end{array}\right.
	\end{equation}
	To conclude about $\det(\nabla u)$, we next study $\nabla \mu_K$. Note that $\mu_K(sx)=s\mu_K(x)$. We invoke Euler’s Homogeneous Function Theorem to conclude
	\begin{equation*}
		x\cdot \nabla\mu_K(x)=\mu_K(x)
	\end{equation*}
	Then, it follows $[x\otimes \nabla\mu_K(x)]x=(x\cdot \nabla\mu_K(x))x=\mu_K(x)x$. Therefore, $$x\in\mbox{Ker}\left(\mu_K(x)I_d-x\otimes \nabla\mu_K(x)\right.$$ Subsequently,
	\begin{equation}
		\det(\nabla u^+(x))=\left\{\begin{array}{cl}
			1&\mbox{ for }\mu_K(x)<1,\\
			0&\mbox{ for }\mu_K(x)>1.
		\end{array}\right.
	\end{equation}
	This completes the proof of Proposition \ref{prop:discont}.
\end{proof}

\bigskip

\noi\textbf{Acknowledgements:} 
AJ would like to thank Anusandhan National Research Foundation (ANRF) for supporting his position as a Ramanujan fellow at HRI, Prayagraj, India, under project file no. RJF/2025/000642. 

%

\end{document}